\documentclass[a4paper,12pt,draft]{amsart}
\usepackage{amsmath, amssymb, amsthm}
\usepackage{braket}
\usepackage{url}
\usepackage{accents}

\theoremstyle{plain}   

\newtheorem{thm}{Theorem}[section]

\newtheorem{Theorem}[thm]{Theorem}

\newtheorem{Lemma}[thm]{Lemma}

\newtheorem{Corollary}[thm]{Corollary}

\newtheorem{Prop}[thm]{Proposition}

\theoremstyle{definition}  

\newtheorem{Remark}[thm]{Remark}
\newtheorem{example}[thm]{例}
\newtheorem{Example}[thm]{Example}

\theoremstyle{definition}  

\newtheorem{Definition}[thm]{Definition}

\newcommand{\KK}{\mathbb{K}}

\newcommand{\BBB}{\mathcal{B}}

\newcommand{\FFF}{\mathcal{F}}
\newcommand{\LLL}{\mathcal{L}}
\newcommand{\TTT}{\mathcal{T}}

\newcommand{\zero}{\boldsymbol{0}}

\newcommand{\Ann}{\operatorname{Ann}}
\newcommand{\rank}{\operatorname{rank}}

\makeatletter
\DeclareSymbolFont{symbolsC}{U}{txsyc}{m}{n}
\DeclareMathSymbol{\MYPerp}{\mathrel}{symbolsC}{121}
\makeatother

\allowdisplaybreaks[3] 

\newbox{\MyTrashBox}
 
\makeatletter
  \def\widebar{\accentset{{\cc@style\underline{\mskip10mu}}}}
  \makeatother

\begin{document}

\title[Hessian matrices of the generating functions for trees]{The eigenvalues of the Hessian matrices of the generating functions for trees with $k$ components}
\author[A. Yazawa]{Akiko Yazawa}
\address[Akiko Yazawa]{Department of Science and Technology,
	Graduate School of Medicine, Science and Technology,
	Shinshu University,
	Matsumoto, Nagano, 390-8621, Japan}
\email{yazawa@math.shinshu-u.ac.jp}
\date{}
\thanks{This work was partly supported by the Sasakawa Scientific Research Grant from       The Japan Science Society. }
\maketitle


\begin{abstract}
Let us consider a truncated matroid $M_{\Gamma}^{r}$ of rank $r$ of a graphic matroid of a graph $\Gamma$. 
The basis for $M_{\Gamma}^{r}$ is the set of the forests with $r$ edges in $\Gamma$. 
We consider this basis generating function and compute its Hessian. 
In this paper, we show that the Hessian of the basis generating function of the truncated matroid of the graphic matroid of the complete or complete bipartite graph does not vanish by calculating the eigenvalues of the Hessian matrix. 
Moreover, we show that the Hessian matrix of the basis generating function of the truncated matroid of the graphic matroid of the complete or complete bipartite graph  has exactly one positive eigenvalue. 
As an application, we show the strong Lefschetz property for the Artinian Gorenstein algebra associated to the truncated matroid. 
\end{abstract}


\section{Introduction}\label{sec:Introduction}

Various applications of the strong Lefschetz property to other areas, e.g. combinatorics, representation theory and so on (see \cite{MR3112920} for details), has been found in the last two decades. 
Recently, algebras associated matroids are studied, e.g. \cite{MR3733101}, \cite{MR3566530}, \cite{MNY}, \cite{NY2019}, \cite{Y2018}. 

In \cite{MR3733101}, Huh and Wang defined a chow ring $A(M)$ associated to a loop-less matroid $M$ on a set $\bar E=\Set{0, 1, \ldots, n}$. 
They show that the ring $A(M)$ has the strong Lefschetz property in the narrow sense (see Definition \ref{SLP}).  
Moreover an element $L$ in $A^{1}(M)$ such that $L$ is strictly submodular function on the family of the subsets of $\bar E$ is a  strong Lefschetz element. 
They also defined algebra $B^*(M)$ associated to a matroid on a set $E=\Set{1, 2, \ldots, n}$ which is a subring of $A(M)$. 
They show that $B^*(M)$ has the ``injective'' Lefschetz property in the case where $M$ is representable. 
Moreover the element $L$ in  $B^1(M)$ such that all the coefficients of all variables are one is a strong Lefschetz element. 

In \cite{MR3566530}, Maeno and Numata defined algebras $Q/J_{M}$ and $A_{M}$ for a matroid $M$
to give an algebraic proof of the Sperner property for the lattice $\LLL(M)$ consisting of flats of $M$. 
The algebra $Q/J_{M}$ is isomorphic to the vector space with basis the set of flats of $M$ as vector spaces. 
The algebra $A_{M}$ is defined to be the quotient algebra of the ring of the differential polynomials by the annihilator of $F_{M}$ (the algebra $A_{M}$ is isomorphic to $B^*(M)$).  
They show that $Q/J_{M}$ has the strong Lefshetz property in the narrow sense if and only if the lattice $\LLL(M)$ is modular, 
and  that $Q/J_{M}=A_{M}$ if and only if $\LLL(M)$ is modular. 
They conjectured that the algebra $A_{M}$ has the strong Lefschetz property for an arbitrary matroid $M$ in an extended abstract \cite{MR2985388} of the paper \cite{MR3566530}. 

In general, a graded Artinian Gorenstein algebra has a representation $A=\KK[x_{1}, x_{2}, \ldots, x_{N}]/\Ann(\Phi)$, where $\Phi$ is a homogeneous polynomial (see Section \ref{sec:Application} for details). 
For a graded Artinian Gorenstein algebra, there is a criterion for the strong Lefschetz property. 
This uses the Hessian matrices (see Theorem \ref{criterion}). 
Roughly, a graded Artinian Gorenstein algebra has the strong Lefschetz property if and only if the Hessians (the determinant of the Hessian matrices) do not vanish. 
Hence the Hessian matrices and Hessians are important. 


In this paper, for the generating function for forests with $k$ components, 
we consider its Hessian matrix and Hessian. 
In \cite{Y2018}, for the generating function for forests with one components (it is called the Kirchhoff polynomial of the complete graph), its Hessian matrix and Hessian are computed. 
The Hessian matrix of the generating function for forests with one components has exactly one positive eigenvalue and its Hessian does not vanish. 
More general, in \cite{NY2019} and \cite{MNY}, the Hessian matrix has exactly one positive eigenvalue and its Hessian does not vanish for the generating function for any simple graphic matroid and any simple matroid, respectively. 
Our main theorem is that for the generating function for forests with $k$ components, 
its Hessian matrix and its Hessian are in the same situation. 
That is the Hessian matrix has exactly one positive eigenvalue and its Hessian does not vanish. 
We gives another proof of the theorem in the case of the truncated matroids of graphic matroids of the complete and complete bipartite graphs by directly calculation of the eigenvalues of the Hessian matrix in \cite{MNY}. 
See also Remark \ref{MNY}.

This paper is organized as follows: 
In Section \ref{sec:Main result}, we consider the generating function for the forests. 
Our main result is that the Hessian of some generating functions for forests does not vanish (Theorem \ref{main theorem} and \ref{main theorem 2}). 
In Section \ref{sec:Application}, we consider the strong Lefschetz property of an algebra associated to a matroid. 
We see a definition of a matroid and its example, and conclude that our main result gives applications to algebras associated to truncated matroids of graphic matroids of the complete and complete bipartite graphs. 

\section{Main result}\label{sec:Main result}
In this section, we show that the Hessian of some generating functions for forests does not vanish (Theorems \ref{main theorem} and \ref{main theorem 2}). 

A \textit{forest} is a graph without cycles. 
Note that a forest is a simple graph. 
For a finite set $V$, define
\begin{align*}
\binom{V}{2}=\Set{\Set{x, y}|x,y\in V, x\neq y}. 
\end{align*}
For a graph $\Gamma$, $V(\Gamma)$ and $E(\Gamma)$ are the set of vertices and edges of $\Gamma$, respectively. 
For a graph $\Gamma$ and an edge $e$, 
$\Gamma\cup e$ stands for a graph 
such that the vertex set is $V(\Gamma)\cup e$, 
and the edge set is $E(\Gamma)\cup \Set{e}$. 
For graphs $\Gamma$ and $\Gamma'$ where $V(\Gamma)\cap V(\Gamma')=\emptyset$, $\Gamma\sqcup \Gamma'$ stands for a graph such that the vertex set is $V(\Gamma)\cup V(\Gamma')$, 
and the edge set is $E(\Gamma)\cup E(\Gamma')$ . 

\subsection{The generating function for the forests in the complete graph}\label{subsec:The generating function for the forests in the complete graph}

For a finite set $V$, we define $\FFF_{V}^{k}$ to be the collection of the forests with the vertex set $V$ and $k$ components.  
We define \textit{the generating function $\Phi_{V, k}$ for $\FFF_{V}^{k}$} by
\begin{align*}
\Phi_{V, k}=\sum_{F\in \FFF_{V}^{k}}\prod_{e\in E(F)}x_{e}. 
\end{align*}

\begin{Remark}\label{the generating function for the forests}
Let $K_{V}$ be the complete graph with the vertex set $V$. 
The set $\FFF_{V}^{k}$ can be regarded as the set of subgraphs of $K_{V}$ with $k$ components.  
A forest with $r$ edges in the complete graph $K_{n}=K_{\Set{1, 2, \ldots, n}}$ is a forest consisting of $n-r$ components.  
For $k=1$, 
an element in $\FFF_{V}^{k}$ is called a \textit{spanning tree} in $K_{V}$, and
the generating function $\Phi_{V, k}$ is called \textit{the Kirchhoff polynomial} of $K_{V}$.  
\end{Remark}

\begin{Example}
Consider $V=\Set{1, 2, 3, 4}$. 
Then, the generating functions are as follows:  
\begin{align*}
\Phi_{V, 3}&=x_{12}+x_{13}+x_{14}+x_{23}+x_{24}+x_{34}, \\
\Phi_{V, 2}&=x_{12}x_{13}+x_{12}x_{14}+x_{12}x_{23}+x_{12}x_{24}+x_{12}x_{34}\\
&+x_{13}x_{14}+x_{13}x_{23}+x_{13}x_{24}+x_{13}x_{34}+x_{14}x_{23}\\
&+x_{14}x_{24}+x_{14}x_{34}+x_{23}x_{24}+x_{23}x_{34}+x_{24}x_{34}, \\
\Phi_{V, 1}&=x_{12}x_{13}x_{23}+x_{12}x_{13}x_{14}+x_{13}x_{14}x_{23}+x_{12}x_{14}x_{23}\\
	&+x_{12}x_{14}x_{24}+x_{12}x_{14}x_{34}+x_{13}x_{23}x_{34}+x_{13}x_{23}x_{24}\\
	&+x_{12}x_{23}x_{24}+x_{14}x_{23}x_{34}+x_{12}x_{13}x_{34}+x_{13}x_{14}x_{24}\\
	&+x_{23}x_{24}x_{34}+x_{12}x_{24}x_{34}+x_{13}x_{24}x_{34}+x_{14}x_{24}x_{34}. 
\end{align*}
\end{Example}

By definition, 
the generating function $\Phi_{V, k}$ is a homogeneous polynomial of degree $\#V-k$. 
Moreover, the generating function $\Phi_{V, k}$ is a sum of square-free monomials. 

For the generating function $\Phi_{V, k}$, consider the matrix
\begin{align*}
H_{\Phi_{V, k}}=\left(\frac{\partial}{\partial x_{e}}\frac{\partial}{\partial x_{e'}}\Phi_{V, k}\right)_{e, e'\in \binom{V}{2}}.  
\end{align*}
The matrix $H_{\Phi_{V, k}}$ is called \textit{the Hessian matrix of $\Phi_{V, k}$}, 
and the determinant $\det H_{\Phi_{V, k}}$ is called \textit{the Hessian of $\Phi_{V, k}$}. 
We define $\widetilde H_{\Phi_{V, k}}$ to be the special value of $H_{\Phi_{V, k}}$ at $x_{e}=1$ for all $e$. 

\begin{Example}
Let $V=\Set{1, 2, 3, 4}$. 
In the case of $\Phi_{V, 2}$, we have 
\begin{align*}
\widetilde H_{\Phi_{V, 2}}=
\begin{pmatrix}
0&1&1&1&1&1 \\
1&0&1&1&1&1 \\
1&1&0&1&1&1 \\
1&1&1&0&1&1 \\
1&1&1&1&0&1 \\
1&1&1&1&1&0 \\
\end{pmatrix}.  
\end{align*}
The eigenvalues of $\widetilde H_{\Phi_{V, 2}}$ are $5, -1, -1, -1, -1, -1$. 

In the case of $\Phi_{V, 1}$, we have 
\begin{align*}
\widetilde H_{\Phi_{V, 1}}=
\begin{pmatrix}
0&3&4&3&3&3 \\
3&0&3&4&3&3 \\
4&3&0&3&3&3 \\
3&4&3&0&3&3 \\
3&3&3&3&0&4 \\
3&3&3&3&4&0 \\
\end{pmatrix}. 
\end{align*}
The eigenvalues of $\widetilde H_{\Phi_{V, 1}}$ are $16,-2,-2,-4,-4,-4$. 
\end{Example}

Note that the matrices $H_{\Phi_{V,1}}$ and $\widetilde H_{\Phi_{V, 1}}$ are always the zero matrices for $V=\Set{1,2,\ldots, n}$. 

\begin{Theorem}[Main theorem]\label{main theorem}
Consider a set $V=\Set{1, 2, \ldots, n}$. 
Let $n\geq 3$ and $0<k<n-2$. 
The determinant $\det\widetilde H_{\Phi_{V, k}}$ does not vanish. 
Moreover, the matrix $\widetilde H_{\Phi_{V, k}}$ has exactly one positive eigenvalue. 
\end{Theorem}

Now, we prove Theorem \ref{main theorem}. 
The first step is to compute the eigenvalues of $\widetilde H_{\Phi_{V, k}}$ (Proposition \ref{eigenvalues}). 
The second step is to show that each eigenvalue of $\widetilde H_{\Phi_{V, k}}$ is non-zero (Proposition \ref{nonzero}). 

The eigenvalues of $\widetilde H_{\Phi_{V, k}}$ are known for $k=1$ \cite[Proposition 3.2]{Y2018}. 
The proof of \cite[Proposition 3.2]{Y2018} can be generalized as follows. 

\begin{Lemma}\label{original eigenvalues}
Let $E$ be the set of the edges of the complete graph $K_{V}$ and $\#V=n$. 
Let $H=\left(h_{e, e'}\right)_{e, e'\in E}$ be the matrix defined by 
\begin{align*}
h_{e, e'}=
\begin{cases}
\alpha, & e=e', \\
\beta, & \#e\cap e'=1, \\
\gamma, & \#e\cap e'=0. 
\end{cases}
\end{align*}
The eigenvalues of $H$ are 
\begin{align*}
\alpha+(2n-4)\beta+\frac{(n-2)(n-3)}{2}\gamma, \\
\alpha-2\beta+\gamma, \\
\alpha+(n-4)\beta-(n-3)\gamma. 
\end{align*}
The dimensions of the eigenspaces of $H$ associate with 
$\alpha+(2n-4)\beta+\frac{(n-2)(n-3)}{2}\gamma, \alpha-2\beta+\gamma$ 
and 
$\alpha+(n-4)\beta-(n-3)\gamma$ are 
\begin{align*}
1, &&
\binom{n}{2}-n, &&
n-1, 
\end{align*}
respectively. 
\end{Lemma}

Fix $V=\Set{1, 2, \ldots, n}$ and $n\geq 4$. 
Define
\begin{align*}
P&=\Set{F\in \FFF_{V}^{k}|\Set{1, 2}, \Set{2, 3}\in E(F)}, \\
R&=\Set{F\in \FFF_{V}^{k}|\Set{1, 2}, \Set{3, 4}\in E(F)},  
\end{align*}
and
\begin{align*}
p=\#P, && q=\#Q. 
\end{align*}

\begin{Lemma}\label{the entries}
We set $\widetilde H_{\Phi_{V, k}}=(h_{e, e'})$. 
Then, we have 
\begin{align*}
h_{e,e'}=
\begin{cases}
0, & e=e',  \\
p, & \#e\cap e'=1,  \\
q, & \#e\cap e'=0. 
\end{cases}
\end{align*}
\end{Lemma}

\begin{proof}
Since $\Phi_{V, k}$ is a sum of square-free monomials, each diagonal component of $\widetilde H_{\Phi_{V, k}}$ is zero. 

Let $V=\Set{1, 2, \ldots, n}$. 
For any $\sigma\in S_{n}$, consider a map $\sigma:V\to V$ such that $i\mapsto \sigma(i)$. 
This map induces an isomorphism between the complete graphs with $n$ vertices. 
For $e, e'\in \binom{V}{2}$ such that $\#e\cap e'=1$, 
the number of the forests with $k$ components which contain the edges $e$ and $e'$ is $p$, since 
there is an isomorphism $\sigma$ such that $\sigma\Set{1,2}=e$ and  $\sigma\Set{2,3}=e'$. 

Similarly, we can prove the case where $\#e\cap e'=0$. 
\end{proof}

\begin{Prop}\label{eigenvalues} 
For $0<k<n-2$,  
the eigenvalues of $\widetilde H_{\Phi_{V, k}}$ are 
\begin{align*}
(2n-4)p+\frac{(n-2)(n-3)}{2}q,&& 
-2p+q, &&
(n-4)p-(n-3)q. 
\end{align*}
The dimensions of the eigenspaces of $\widetilde H_{\Phi_{V, k}}$ associate with $(2n-4)p+\frac{(n-2)(n-3)}{2}q$, $-2p+q$ and $(n-4)p-(n-3)q$ are 
\begin{align*}
1, &&
\binom{n}{2}-n, &&
n-1, 
\end{align*}
respectively. 
\end{Prop}

Next, we show that each eigenvalue of $\widetilde H_{\Phi_{V, k}}$ is non-zero. 
Let $W$ be a subset of $V$ such that $\Set{1,2,3,4}\subset W$. 
For $W$, define 
\begin{align*}
\TTT_{W}'&=\Set{T\in \FFF_{W}^{1}| \Set{1,2}, \Set{2,3}\in E(T)}, \\
\TTT_{W}''&=\Set{T\in \FFF_{W}^{1}| \Set{1,2}, \Set{3,4}\in E(T)}.  
\end{align*}
Due to \cite{MR0231755}, we have 
\begin{align}\label{Moon}
\#\TTT_{W}'=3\#W^{\#W-4}, &&\#\TTT_{W}''=4\#W^{\#W-4}. 
\end{align}

For $W$, define
\begin{align*}
\FFF_{W}'&=\Set{T\sqcup T' \in \FFF_{W}^{2}| \Set{1,2}, \Set{2,3}\in E(T), 4\in V(T')},  \\
\FFF_{W}''&=\Set{T\sqcup T' \in \FFF_{W}^{2}| \Set{1,2}\in E(T), \Set{3, 4}\in E(T')}. 
\end{align*}

\begin{Lemma}\label{forestbij}
For $\Set{1,2,3,4}\subset W \subset \Set{1, 2, \ldots, n}$, we have 
\begin{align*}
\#\FFF_{W}'=\#\FFF_{W}''. 
\end{align*}
\end{Lemma}

\begin{proof}
We construct a bijection between $\FFF_{W}'$ and $\FFF_{W}''$. 

We define a map $f$ from $\FFF_{W}'$ to $\FFF_{W}''$ in the following manner: 
Let $T=T_{a}\sqcup T_{b}\in \FFF_{W}'$, 
where $T_{a}$ contains the edges $\Set{1,2}$ and $\Set{2,3}$, 
and $T_{b}$ contains the vertices $4$. 
Since $T_{a}$ is a tree, if we remove the edge $\Set{2,3}$ from $T_{a}$, 
then the tree $T_{a}$ is decomposed into two trees. 
One of them contains the edges $\Set{1,2}$, and we set $T_{a}^{\Set{1, 2}}$ for this tree. 
The other one contains the vertex $3$, and we set $T_{a}^{\Set{3}}$ for this tree. 
Then we have a decomposition $T_{a}=T_{a}^{\Set{1, 2}}\cup\Set{2,3}\cup T_{a}^{\Set{3}}$. 
For $T$, define
\begin{align*}
f(T)=T_{a}^{\Set{1,2}}\sqcup\left(\Set{3, 4}\cup T_{a}^{\Set{3}}\cup T_{b}\right). 
\end{align*} 
Then, we have $f(T)\in \FFF_{W}''$. 
Hence, the map $f$ is well-defined. 

We define a map $g$ from $\FFF_{W}''$ to $\FFF_{W}'$ in the following manner: 
Let $T=T_{c}\sqcup T_{d}\in \FFF_{W}''$, 
where $T_{c}$ contains the edge $\Set{1,2}$, 
and $T_{d}$ contains the edge $\Set{3,4}$. 
Since $T_{d}$ is a tree, if we remove the edge $\Set{3,4}$ from $T_{d}$, 
then the tree $T_{d}$ is decomposed into two trees. 
One of them contains the vertex $3$, and we set $T_{d}^{\Set{3}}$ for this tree. 
The other one contains the vertex $4$, and we set $T_{d}^{\Set{4}}$ for this tree. 
Then we have a decomposition $T_{d}=T_{d}^{\Set{3}}\cup\Set{3,4}\cup T_{d}^{\Set{4}}$. 
For $T$, define
\begin{align*}
g(T)=\left(\Set{2,3}\cup T_{d}^{\Set{3}}\cup T_{c}\right)\sqcup\left(T_{d}^{\Set{4}}\right). 
\end{align*} 
Then, we have $g(T)\in \FFF_{W}'$. 
Hence, the map $g$ is well-defined. 

The maps $f$ and $g$ are inverses of each other.  
\end{proof}

We are ready to show that each eigenvalue of $\widetilde H_{\Phi_{V, k}}$ is non-zero. 

\begin{Prop}\label{nonzero}
Let $\#V=n (n\geq 3)$ and $0<k<n-2$. 
The matrix $\widetilde H_{\Phi_{V, k}}$ does not have the zero-eigenvalues. 
Moreover we have the following: 
\begin{align*}
(2n-4)p+\frac{(n-2)(n-3)}{2}q&>0, \\
-2p+q&<0, \\
(n-4)p-(n-3)q&<0. 
\end{align*}
\end{Prop}

\begin{proof}
Since $p$ and $q$ are the number of some forests, we have $p, q>0$. 
Therefore $(2n-4)p+\frac{(n-2)(n-3)}{2}q$, the eigenvalue of $\widetilde H_{\Phi_{V, k}}$, is positive.  

Let us show that the other eigenvalues of $\widetilde H_{\Phi_{V, k}}$ are negative. 
Let 
\begin{align*}
P&=\left(\bigsqcup_{W}(\TTT_{W}'\times \FFF_{W^{\mathsf{c}}}^{(k-1)})\sqcup\bigsqcup_{W}(\FFF_{W}'\times\FFF_{W^{\mathsf{c}}}^{(k-2)})\right), \\
Q&=\left(\bigsqcup_{W}(\TTT_{W}''\times \FFF_{W^{\mathsf{c}}}^{(k-1)})\sqcup\bigsqcup_{W}(\FFF_{W}''\times\FFF_{W^{\mathsf{c}}}^{(k-2)})\right),  
\end{align*}
where the sums run over $\Set{1,2,3,4}\subset W \subset \Set{1, 2, \ldots, n}$. 
Then we have $\# P=p$ and $\# Q=q$. 
For $W$, define $f_{W}=\#\FFF_{W}'$ and $t_{W}=\#W^{\#W-4}$. 
By \eqref{Moon} and Lemma \ref{forestbij}, 
we have 
\begin{align*}
\#\TTT_{W}'=3t_{W},  &&
\#\TTT_{W}''=4t_{W},  &&
\#\FFF_{W}'=\#\FFF_{W}''=f_{W}. 
\end{align*}
Then 
\begin{align*}
p&=\sum_{W}3t_{W}\#\FFF_{W^{\mathsf{c}}}^{(k-1)}+\sum_{W}f_{W}\#\FFF_{W^{\mathsf{c}}}^{(k-2)} \\
&=3\sum_{W}t_{W}\#\FFF_{W^{\mathsf{c}}}^{(k-1)}+\sum_{W}f_{W}\#\FFF_{W^{\mathsf{c}}}^{(k-2)}, 
\end{align*}
and
\begin{align*}
q&=\sum_{W}4t_{W}\#\FFF_{W^{\mathsf{c}}}^{(k-1)}+\sum_{W}f_{W}\#\FFF_{W^{\mathsf{c}}}^{(k-2)} \\
&=4\sum_{W}t_{W}\#\FFF_{W^{\mathsf{c}}}^{(k-1)}+\sum_{W}f_{W}\#\FFF_{W^{\mathsf{c}}}^{(k-2)}.  
\end{align*}
If we set $t=\sum_{W}t_{W}\#\FFF_{W^{\mathsf{c}}}^{(k-1)}$ and $f=\sum_{W}f_{W}\#\FFF_{W^{\mathsf{c}}}^{(k-2)}$, 
then we have 
\begin{align*}
p=&3 t+f, \\
q=&4 t+f. 
\end{align*}
Note that $t, f>0$. 
Hence 
\begin{align*}
-2p+q&=-2(3 t+f)+(4 t+f) \\
&=-2 t-f \\
&<0, \\
(n-4)p-(n-3)q&=(n-4)(3 t+f)-(n-3)(4 t+f) \\
&=-n t-f \\
&<0. 
\end{align*}
\end{proof}

By Propositions \ref{eigenvalues} and \ref{nonzero}, we have Theorem \ref{main theorem}.

\subsection{The generating function for the forests in the complete bipartite graph}\label{subsec:The generating function for the forests in the complete bipartite graph}

\newcommand{\vA}{{1}}
\newcommand{\vB}{\bar{1}}
\newcommand{\vC}{{2}}
\newcommand{\vD}{\bar{2}}
\newcommand{\vM}{{m}}
\newcommand{\vN}{\bar{n}}
\newcommand{\vX}{{i}}
\newcommand{\vY}{\bar{j}}

Let $X$ and $Y$ be finite sets and $X\cap Y=\emptyset$. 
Let $V=X\sqcup Y$. 
For $X$ and $Y$, define
\begin{align*}
\FFF_{X, Y}^{k}=\Set{F\in \FFF_{V}^{k}|\text{if $e\in E(F)$, then $e\not\in\binom{X}{2}$ and $e\not\in\binom{Y}{2}$}},  
\end{align*}
where $\FFF_{V}^{k}$ is in \ref{subsec:The generating function for the forests in the complete graph}. 
We define \textit{the generating function $\Phi_{X, Y, k}$ for $\FFF_{X, Y}^{k}$} by
\begin{align*}
\Phi_{X, Y, k}=\sum_{F\in \FFF_{X, Y}^{k}}\prod_{e\in E(F)}x_{e}. 
\end{align*}

\begin{Remark}\label{the generating function for the forests 2}
Let $K_{X, Y}$ be the complete bipartite graph with the vertex sets $X$ and $Y$. 
The set $\FFF_{X, Y}^{k}$ can be regarded as the set of forests of $K_{X, Y}$ with $k$ components.  
A forest with $r$ edges in the complete bipertite graph $K_{m, n}=K_{\Set{1, 2, \ldots, m}, \Set{\bar{1},\ldots, \bar{n}}}$ is a forest consisting of $m+n-r$ components.  
For $k=1$, 
an element in $\FFF_{X, Y}^{k}$ is called a \textit{spanning tree} in $K_{X, Y}$, and
the generating function $\Phi_{X, Y, k}$ is called \textit{the Kirchhoff polynomial} of $K_{X, Y}$.  
\end{Remark}

\begin{example}
Let $X=\Set{1,2}$ and $Y=\Set{\vB, \vD}$. 
Then, the generating functions are as follows:  
\begin{align*}
\Phi_{X,Y, 3}&=x_{1\vB}+x_{1\vD}+x_{2\vB}+x_{2\vD}, \\
\Phi_{X,Y, 2}&=x_{1\vB}x_{1\vD}+x_{1\vB}x_{2\vB}+x_{1\vB}x_{2\vD}+x_{1\vD}x_{2\vB}+x_{1\vD}x_{2\vD}+x_{2\vB}x_{2\vD}, \\
\Phi_{X,Y, 1}&=x_{1\vB}x_{1\vD}x_{2\vB}+x_{1\vB}x_{1\vD}x_{2\vD}+x_{1\vB}x_{2\vB}x_{2\vD}+x_{1\vD}x_{2\vB}x_{2\vD}. 
\end{align*}
\end{example}

By definition, 
the generating function $\Phi_{X, Y, k}$ is a homogeneous polynomial of degree $\#X+\#Y-k$. 
Moreover, the generating function $\Phi_{X, Y, k}$ is a sum of square-free monomials. 

Let us consider the Hessian matrix $H_{\Phi_{X, Y, k}}$ and Hessian $\det H_{\Phi_{X, Y, k}}$ of the generating function $\Phi_{X, Y, k}$. 
We define $\widetilde H_{\Phi_{X, Y, k}}$ to be the special value of $H_{\Phi_{X, Y, k}}$ at $x_{e}=1$ for all $e$. 

\begin{example}
Let $X=\Set{1,2}$ and $Y=\Set{\vB, \vD}$. 
In the case of $\Phi_{X,Y, 2}$, we have 
\begin{align*}
\widetilde H_{\Phi_{X,Y, 2}}=
\begin{pmatrix}
0&1&1&1 \\
1&0&1&1 \\
1&1&0&1 \\
1&1&1&0 \\
\end{pmatrix}.  
\end{align*}
The eigenvalues of $\widetilde H_{\Phi_{X,Y, 2}}$ are $3, -1, -1, -1$. 

In the case of $\Phi_{X,Y, 1}$, we have 
\begin{align*}
\widetilde H_{\Phi_{X,Y, 1}}=
\begin{pmatrix}
0&2&2&2 \\
2&0&2&2 \\
2&2&0&2 \\
2&2&2&0 \\
\end{pmatrix}.  
\end{align*}
The eigenvalues of $\widetilde H_{\Phi_{X,Y, 1}}$ are $6,-2,-2,-2$. 
\end{example}

Note that the matrices $H_{\Phi_{X,Y,1}}$ and $\widetilde H_{\Phi_{X,Y, 1}}$ are always the zero matrices for any $X$ and $Y$.

\begin{Theorem}[Main theorem]\label{main theorem 2}
Consider sets $X$ and $Y$ such that $X\cap Y=\emptyset$, $\#X\geq 2$ and $\#Y\geq 2$.  
For $0<k<\#X+\#Y-2$,  
the determinant $\det\widetilde H_{\Phi_{X, Y, k}}$ does not vanish. 
Moreover, the matrix $\widetilde H_{\Phi_{X, Y, k}}$ has exactly one positive eigenvalue. 
\end{Theorem}

Now, we prove Theorem \ref{main theorem 2}. 
The first step is to compute the eigenvalues of $\widetilde H_{\Phi_{X, Y, k}}$ (Proposition \ref{eigenvalues 2}). 
The second step is to show that each eigenvalue of $\widetilde H_{\Phi_{X, Y, k}}$ is non-zero (Proposition \ref{nonzero 2}). 

The eigenvalues of $\widetilde H_{\Phi_{X, Y, k}}$ are known for $k=1$ \cite[Proposition 3.15]{Y2018}. 
The proof of \cite[Proposition 3.15]{Y2018} can be generalized as follows. 

\begin{Lemma}\label{original eigenvalues 2}
Let $E$ be the set of the edges of the complete bipartite graph $K_{X, Y}$, $\#X=m$ and $\#Y=n$. 
Let $H=\left(h_{e, e'}\right)_{e, e'\in E}$ be the matrix defined by 
\begin{align*}
h_{e, e'}=
\begin{cases}
\alpha, & e=e', \\
\beta, & e\cap e'\in X, \\
\gamma, & e\cap e'\in Y, \\
\delta, & e\cap e'=\emptyset. 
\end{cases}
\end{align*}
The eigenvalues of $H$ are 
\begin{align*}
\alpha+(n-1)\beta+(m-1)\gamma+(m-1)(n-1)\delta,\\
\alpha+(n-1)\beta-\gamma-(n-1)\delta, \\
\alpha-\beta+(m-1)\gamma-(m-1)\delta, \\
\alpha-\beta-\gamma+\delta. 
\end{align*}
The dimensions of the eigenspaces of $H$ associate with 
$
\alpha+(n-1)\beta+(m-1)\gamma+(m-1)(n-1)\delta,
\alpha+(n-1)\beta-\gamma-(n-1)\delta, 
\alpha-\beta+(m-1)\gamma-(m-1)\delta
$ 
and 
$\alpha-\beta-\gamma+\delta
$ 
are
\begin{align*}
1, &&
m-1, &&
n-1, &&
(m-1)(n-1),  
\end{align*}
respectively. 
\end{Lemma} 

Fix 
\begin{align*}
X&=\Set{\vA, \vC, \ldots, \vM}, \\
Y&=\Set{\vB, \vD, \ldots, \vN}, 
\end{align*}
and $m, n\geq 2$. 
By definition, we have $\#X=m$ and $\#Y=n$. 
Define
\begin{align*}
P&=\Set{F\in\FFF_{X, Y}^{k}|\Set{\vA, \vB}, \Set{\vA, \vD}\in E(F)}, \\
Q&=\Set{F\in\FFF_{X, Y}^{k}|\Set{\vA, \vB}, \Set{\vB, \vC}\in E(F)}, \\
R&=\Set{F\in\FFF_{X, Y}^{k}|\Set{\vA, \vB}, \Set{\vC, \vD}\in E(F)},  
\end{align*}
and 
\begin{align*}
p=\#P, && q=\#Q, && r=\#R. 
\end{align*}

\begin{Lemma}\label{the entries 2}
We set $\widetilde H_{\Phi_{X, Y, k}}=(h_{e, e'})$. 
Then, we have 
\begin{align*}
h_{e,e'}=
\begin{cases}
0, & e=e',  \\
p, & e\cap e'\in X,  \\
q, & e\cap e'\in Y,  \\
r, & e\cap e'=\emptyset. 
\end{cases}
\end{align*}
\end{Lemma}

\begin{proof}
Since $\Phi_{X, Y, k}$ is a sum of square-free monomials, each diagonal component of $\widetilde H_{\Phi_{X, Y, k}}$ is zero. 

For any $(\sigma, \tau)\in S_{m}\times S_{n}$, consider a map on the vertex set $X\sqcup Y$ such that 
\begin{align*}
X\ni\vX\mapsto \sigma(\vX), && Y\ni\vY\mapsto \overline{\tau(j)}. 
\end{align*}
The map induces an automorphism of $K_{X,Y}$. 
Similarly to Lemma \ref{the entries}, we can prove Lemma \ref{the entries 2}. 
\end{proof}

As a corollary of Lemma \ref{original eigenvalues 2}, we obtain the following. 

\begin{Prop}\label{eigenvalues 2}
For $0<k<m+n-2$, 
the eigenvalues of $\widetilde H_{\Phi_{X, Y, k}}$ are 
\begin{align*}
(n-1)p+(m-1)q+(m-1)(n-1)r,\\
(n-1)p-q-(n-1)r, \\
-p+(m-1)q-(m-1)r, \\
-p-q+r. 
\end{align*}
The dimensions of the eigenspaces of $\widetilde H_{\Phi_{X, Y, k}}$ associate with 
$(n-1)p+(m-1)q+(m-1)(n-1)r, (n-1)p-q-(n-1)r,-p+(m-1)q-(m-1)r$ and $-p-q+r$ are
\begin{align*}
1, &&
m-1, &&
n-1, &&
(m-1)(n-1),  
\end{align*}
respectively. 
\end{Prop}

Next, we show that each eigenvalue of $\widetilde H_{\Phi_{X, Y, k}}$ is non-zero. 
Define  
\begin{align*}
Z&=\Set{F\in \FFF_{X, Y}^{k}|\Set{\vA, \vB}, \Set{\vA, \vD}, \Set{\vC,\vD}\in E(F)}, \\
P'&=\Set{F\in \FFF_{X, Y}^{k}|\Set{\vA, \vB}, \Set{\vA, \vD}\in E(F), \Set{\vC,\vD}\not\in E(F)}, \\
R'&=\Set{F\in \FFF_{X, Y}^{k}|\Set{\vA, \vB}, \Set{\vC, \vD}\in E(F), \Set{\vA,\vD}\not\in E(F)}. 
\end{align*}
Note that $Z=P\cap R$, $P'=P\setminus Z$ and $R'=R\setminus Z$. 
Then, the sets $P$ and $R$ are decomposed into
\begin{align*}
P=Z\sqcup P', &&
R=Z\sqcup R', 
\end{align*}
respectively. 
Let $F\in P'$. 
Consider the component $T$ of $F$ with the vertex $\vA$.  
Note that the vertices $\vB$ and $\vD$ are in $T$. 
If we remove the edges $\Set{\vA, \vB}$ and $\Set{\vA, \vD}$ from $T$, 
then the tree $T$ is decomposed into three trees. 
One of them contains the vertex $\vA$, denoted by $F_{\vA}$. 
One of them contains the vertex $\vB$, denoted by $\overline{F_{1}}$. 
One of them contains the vertex $\vD$, denoted by $\overline{F_{2}}$. 
Let 
\begin{align*}
P_{1}&=\Set{F\in P'|\vC\in F_{\vA}},  \\
P_{2}&=\Set{F\in P'|\vC\in \overline{F_{1}}},  \\
P_{3}&=\Set{F\in P'|\vC\not\in F_{\vA}, \vC\not\in \overline{F_{1}}, \vC\not\in \overline{F_{2}}}, \\ 
P_{4}&=\Set{F\in P'|\vC\in \overline{F_{2}}}. 
\end{align*}
Then we have a decomposition
\begin{align*}
P=Z\sqcup P_{1}\sqcup P_{2}\sqcup P_{3}\sqcup P_{4}. 
\end{align*}
Let $F\in R'$. 
Consider the component $T$ of $F$ with the vertex $\vA$. 
Note that $\vB\in T$. 
If we remove the edge $\Set{\vA, \vB}$ from $T$, 
then the tree $T$ is decomposed into two trees. 
One of them contains the vertex $\vA$, denoted by $F_{\vA}$. 
The other one contains the vertex $\vB$, denoted by $\overline{F_{1}}$. 
Let 
\begin{align*}
R_{1}&=\Set{F\in R'|\vC\in F_{\vA}, \vD\not\in F_{\vA}, \vD\not\in \overline{F_{1}}},  \\
R_{2}&=\Set{F\in R'|\vC\in \overline{F_{1}}, \vD\not\in F_{\vA}, \vD\not\in \overline{F_{1}}},  \\
R_{3}&=\Set{F\in R'|\vC\not\in F_{\vA}, \vC\not\in \overline{F_{1}}, \vD\not\in F_{\vA}, \vD\not\in \overline{F_{1}}}, \\
R_{4}&=\Set{F\in R'|\vC\not\in F_{\vA}, \vC\not\in \overline{{F_{1}}}, \vD\in F_{\vA}}, \\
R_{5}&=\Set{F\in R'|\vC\not\in F_{\vA}, \vC\not\in \overline{F_{1}}, \vD\in \overline{F_{1}}}. 
\end{align*}
Then we have a decomposition
\begin{align*}
R=Z\sqcup R_{1}\sqcup R_{2}\sqcup R_{3}\sqcup R_{4}\sqcup R_{5}. 
\end{align*}

\begin{Lemma}\label{pr123}
We have $\#P_{1}=\#R_{1}$, $\#P_{2}=\#R_{2}$ and $\#P_{3}=\#R_{3}$. 
\end{Lemma}

\begin{proof}
For $1\leq i\leq 3$, we define each map $f_{i}$ from $P_{i}$ to $R_{i}$ in the following manner: 
Let $F\in P_{i}$. 
Define $f_{i}(F)$ to be the forest removing the edge $\Set{\vA, \vD}$ from $F$ and adding the edge $\Set{\vB,\vC}$. 
Then $f_{i}(F)\in R_{i}$. 
We define each map $g_{i}$ from $R_{i}$ to $P_{i}$ in the following manner: 
Let $F\in R_{i}$. 
Define $g_{i}(F)$ to be the forest removing the edge $\Set{\vB, \vC}$ from $F$ and adding the edge $\Set{\vA,\vD}$. 
Then $g_{i}(F)\in P_{i}$. 
The maps $f_{i}$ and $g_{i}$ are inverses of each other.  
\end{proof}

\begin{Lemma}\label{pr4}
We have $\#P_{4}=\#R_{4}$. 
\end{Lemma}

\begin{proof}
We define a map $h$ from $P_{4}$ to $R_{4}$ in the following manner: 
Let $F\in P_{4}$. 
Let $F'$ be the forest such that transpose the vertices $\vA$ and $\vC$ of $F$. 
Note that the vertices $\vA$ and $\vD$ in $F'$ are not connected by an edge since the vertices $\vC$ and $\vD$ in $F$ are not connected by an edge. 
Define $h(F)$ to be the forest removing the edge $\Set{\vB, \vC}$ from $F'$ and adding the edge $\Set{\vA,\vB}$. 
Then $h(F)\in R_{4}$. 
We define a map $h'$ from $R_{4}$ to $P_{4}$ in the following manner: 
Let $F\in R_{4}$. 
Let $F'$ be the forest such that transpose the vertices $\vA$ and $\vC$ of $F$. 
Note that the vertices $\vC$ and $\vD$ in $F'$ are not connected by an edge since the vertices $\vA$ and $\vD$ in $F$ are not connected by an edge. 
Define $h'(F)$ to be the forest removing the edge $\Set{\vB, \vC}$ from $F'$ and adding the edge $\Set{\vA,\vB}$. 
Then $h'(F)\in P_{4}$. 
The maps $h$ and $h'$ are inverses of each other. 
\end{proof}

We obtain from Lemmas \ref{pr123} and \ref{pr4} the following. 

\begin{Lemma}\label{pr}
We have $\#P<\#R$. 
\end{Lemma}

Define  
\begin{align*}
Z'&=\Set{F\in \FFF_{X, Y}^{k}|\Set{\vA, \vB}, \Set{\vB, \vC}, \Set{\vC,\vD}\in E(F)}, \\
Q'&=\Set{F\in \FFF_{X, Y}^{k}|\Set{\vA, \vB}, \Set{\vB, \vC}\in E(F), \Set{\vC,\vD}\not\in E(F)}, \\
R''&=\Set{F\in \FFF_{X, Y}^{k}|\Set{\vA, \vB}, \Set{\vC, \vD}\in E(F), \Set{\vB,\vC}\not\in E(F)}. 
\end{align*}
Note that $Z'=Q\cap R$, $Q'=Q\setminus Z'$ and $R''=R\setminus Z'$. 
Then, the sets $Q$ and $R$ are decomposed into
\begin{align*}
Q=Z'\sqcup Q', &&
R=Z'\sqcup R'', 
\end{align*}
respectively. 
Let $F\in Q'$. 
Consider the component $T$ of $F$ with the vertex $\vA$.  
Note that the vertices $\vB$ and $\vC$ are in $T$. 
If we remove the edges $\Set{\vA, \vB}$ and $\Set{\vB, \vC}$ from $T$, 
then the tree $T$ is decomposed into three trees. 
One of them contains the vertex $\vA$, denoted by $F_{\vA}$. 
One of them contains the vertex $\vB$, denoted by $\overline{F_{1}}$. 
One of them contains the vertex $\vC$, denoted by $F_{\vC}$. 
Let 
\begin{align*}
Q_{1}&=\Set{F\in Q'|\vD\in F_{\vA}},  \\
Q_{2}&=\Set{F\in Q'|\vD\in \overline{F_{1}}},  \\
Q_{3}&=\Set{F\in Q'|\vD\not\in F_{\vA}, \vC\not\in \overline{F_{1}}, \vC\not\in F_{\vC}}, \\ 
Q_{4}&=\Set{F\in Q'|\vD\in F_{\vC}}. 
\end{align*}
Then we have a decomposition
\begin{align*}
Q=Z'\sqcup Q_{1}\sqcup Q_{2}\sqcup Q_{3}\sqcup Q_{4}. 
\end{align*}
Let $R'_{i}=R_{i}\cap R''$ for $1\leq i \leq 5$. 
Then we have a decomposition
\begin{align*}
R=Z'\sqcup R'_{1}\sqcup R'_{2}\sqcup R'_{3}\sqcup R'_{4}\sqcup R'_{5}. 
\end{align*}
Similarly to Lemma \ref{pr}, we obtain the following. 

\begin{Lemma}\label{qr'}
We have $\#Q<\#R$. 
\end{Lemma}

Similarly to Lemma \ref{pr123}, the number of elements of $Q_{2}$ and $R'_{2}$ are the same. 
We, however, consider a relation between $Q_{2}$ and $R_{5}$ in the following. 

\begin{Lemma}\label{q2r5}
We have $\#Q_{2}=\#R_{5}$. 
\end{Lemma}

\begin{proof}
We define a map $f'$ from $Q_{2}$ to $R_{5}$ in the following manner: 
Let $F\in Q_{2}$. 
Define $f'(F)$ to be the forest removing the edge $\Set{\vB, \vC}$ from $F$ and adding the edge $\Set{\vC, \vD}$. 
Then $f'(F)\in R_{5}$. 
We define a map $g'$ from $R_{5}$ to $Q_{2}$ in the following manner: 
Let $F\in R_{5}$. 
Define $g'(F)$ to be the forest removing the edge $\Set{\vC, \vD}$ from $F$ and adding the edge $\Set{\vB, \vC}$. 
Then $g'(F)\in Q_{2}$. 
The maps $f'$ and $g'$ are inverses of each other. 
\end{proof}

Combining Lemmas \ref{pr123}, \ref{pr4} and \ref{q2r5}, we obtain the following. 

\begin{Lemma}\label{pqr}
We have $\#R<\#P+\#Q$. 
\end{Lemma}

Summarize Lemmas \ref{pr}, \ref{qr'} and \ref{pqr},  
and we obtain the following. 

\begin{Lemma}\label{prqr}
We have
\begin{align*}
p-r<0, && q-r<0, && -p-q+r<0.  
\end{align*}
\end{Lemma}

We are ready to show that each eigenvalue of $\widetilde H_{\Phi_{X, Y, k}}$ is non-zero. 

\begin{Prop}\label{nonzero 2}
Let $0<k<m+n-2$. 
The matrix $\widetilde H_{\Phi_{X, Y, k}}$ does not have the zero-eigenvalues. 
Moreover we have the following: 
\begin{align*}
(n-1)p+(m-1)q+(m-1)(n-1)r>0,\\
(n-1)p-q-(n-1)r<0, \\
-p+(m-1)q-(m-1)r<0, \\
-p-q+r<0. 
\end{align*}
\end{Prop}

\begin{proof}
Since $p, q$ and $r$ are the number of some forests, we have $p, q, r>0$. 
Therefore $(n-1)p+(m-1)q+(m-1)(n-1)r$, the eigenvalue of $\widetilde H_{\Phi_{X, Y, k}}$, is positive.  

Let us show that the other eigenvalues of $\widetilde H_{\Phi_{X, Y, k}}$ are negative. 
For the other eigenvalues, we have
\begin{align*}
(n-1)p-q-(n-1)r&=(p-r)n+(-p-q+r), \\
-p+(m-1)q-(m-1)r&=(q-r)m+(-p-q+r). 
\end{align*}
It follows from Lemma \ref{prqr} that the other eigenvalues are negative. 
\end{proof}

By Propositions \ref{eigenvalues 2} and \ref{nonzero 2}, we have Theorem \ref{main theorem 2}. 

\section{Application}\label{sec:Application}

In this section, we consider the strong Lefschetz property of a graded Artinian Gorenstein algebra associated to a matroid, which is defined by Maeno and Numata in \cite{MR3566530}. 
They showed that the strong Lefschetz property for the algebra associated to the uniform matroid \cite{MN32012}.  
Here, we discuss the Lefschetz property of the algebra associated to the truncated matroids of the graphic matroids of the complete and complete bipartite graphs.  

First of all, we recall definitions of matroids and the strong Lefschetz property. 

A {\it matroid} $M$ is an ordered pair $(E, \BBB)$ consisting of a finite set $E$ and a collection $\BBB$ of subsets of $E$ satisfying the following properties: 
\begin{itemize}
\item $\BBB\neq\emptyset$.
\item If $B_1$ and $B_2$ are in $\BBB$ and $x\in B_1\setminus B_2$, then there is an element $y\in B_2\setminus B_1$ such that $\{y\}\cup(B_1\setminus\{x\})\in\BBB$. 
\end{itemize}
In this case, we call each $B\in\BBB$ a {\it basis} of $M$ and $E$ the \textit{ground set} of $M$. 

\begin{Prop}\label{the number of basis are the same}
Let $M$ be a matroid with the basis set $\BBB$. 
If $B$ and $B'$ are basis of $M$, then the number of elements of them are the same. 
In other words, if $B, B'\in \BBB$, then $\#B=\#B'$. 
\end{Prop}

We say that a matroid $M$ has \textit{rank} $r$ if the number of elements of a basis of $M$ is $r$.  
The rank of $M$ is denoted by $\rank M$. 

\begin{Example}
We see some examples of matroids. 
\begin{itemize}
\item[(a)] For any finite graph $\Gamma=(V, E)$ with the vertex set $V$ and the edge set $E$, we call a subgraph $T\subseteq \Gamma$ a {\it spanning tree} in $\Gamma$ if $T$ does not contain any cycles and $T$ passes through all vertices of $\Gamma$. 
Let $\BBB_\Gamma$ be the set of all spanning trees in $\Gamma$. Then $M(\Gamma)=(E, \BBB_\Gamma)$ is a matroid. 
In this case, $\rank M_{\Gamma}=\#V-1$. 
These matroids are called {\it graphic matroids}.  
\item[(d)] Let $M=(E, \BBB)$ be a matroid and
\begin{align*}
\BBB_{r}=\Set{B'\in \binom{E}{r}|\text{there exists $B\in \BBB$ such that $B'\subset B$}}. 
\end{align*}
Then $M=(E, \BBB_{r})$ is a matroid. 
In this case, $\rank M=r$. 
These matroids are called {\it truncated matroids} of $M$. 
\end{itemize}
\end{Example}

Let $M$ be a matroid with the ground set $E$ and $\BBB$ the set of basis for $M$. 
For $M$, define 
\begin{align*}
\Phi_{M}=\sum_{B\in \BBB}\prod_{b\in B}x_{b}. 
\end{align*}
We call $\Phi_{M}$ the basis generating function of $M$. 
By Proposition \ref{the number of basis are the same}, 
for a matroid $M=(E, \BBB)$ of rank $r$, 
its basis generating function $\Phi_{M}$ is a homogeneous polynomial of degree $r$ in $|E|$ variables with positive coefficients.


\begin{Remark}\label{the basis generating function for the truncated matroid of the graphic matroid of the complete graph}
Let $M_{K_{n}}^{r}$ be the truncated matroid of rank $r$ of the graphic matroid $M_{K_{n}}$ of the complete graph $K_{n}$. 
Its bases are the forests with $r$ edges. 
Hence, its basis generating function is $\Phi_{n-r}$ in Section \ref{sec:Main result}. 
\end{Remark}


\begin{Definition}\label{SLP}
Let $A=\bigoplus_{k=0}^{s} A_{k}$, $A_{s}\neq \zero$, 
be a graded Artinian algebra. 
We say that $A$ has the \emph{strong Lefschetz property} 
if there exists an element $L \in A_{1}$ such that the multiplication map $\times L^{s-2k}\colon A_{k}\to A_{s-k}$ is bijective for all $k\leq \frac{s}{2}$. 
We call $L \in A_{1}$ with this property a \emph{strong Lefschetz element}. 
\end{Definition}

Let $\KK$ be a field of characteristic zero.  
For a homogeneous polynomial $\Phi\in \KK[x_{1}, x_{2}, \ldots, x_{N}]$, we define $\Ann(\Phi)$ by
\begin{align*}
\Ann(\Phi)=\Set{P\in \KK[x_{1}, \ldots, x_{N}]| P\left(\frac{\partial}{\partial x_{1}}, 
 \ldots, \frac{\partial}{\partial x_{N}}\right)\Phi=0}. 
\end{align*}
Then $\Ann(\Phi)$ is a homogeneous ideal of $\KK[x_{1}, \ldots, x_{N}]$. 
We consider $A=\KK[x_{1}, \ldots, x_{N}]/\Ann(\Phi)$. 
Since $\Ann(\Phi)$ is homogeneous, the algebra $A$ is graded. 
Furthermore $A$ is an Artinian Gorenstein algebra. 
Conversely, a graded Artinian Gorenstein algebra $A$ has the presentation
\begin{align*}
A=\KK[x_{1}, \ldots, x_{N}]/\Ann(\Phi)
\end{align*}
for some homogeneous polynomial $\Phi\in \KK[x_{1}, x_{2}, \ldots, x_{N}]$. 
We decompose $A$ into the homogeneous spaces $A_{k}$. 
Then $A_{k}$ is a vector space over $\KK$ for all $k$.  
Let $\Lambda_{k}$ be the basis for $A_{k}$. 
We define the matrix $H_{\Phi}^{(k)}$ by 
\begin{align*}
H_{\Phi}^{(k)}=
\left(
e_{i}\left(\frac{\partial}{\partial x_{1}},\ldots, \frac{\partial}{\partial x_{N}}\right)
e_{j}\left(\frac{\partial}{\partial x_{1}},\ldots, \frac{\partial}{\partial x_{N}}\right)
\Phi
\right)_{e_{i},e_{j}\in \Lambda_{k}}. 
\end{align*}
The determinant of $H_{\Phi}^{(k)}$ is called the $k$th \emph{Hessian} of $\Phi$ with respect to the basis $\Lambda_{k}$. 

\begin{Remark}\label{0th Hessian}
Since $A_{0}\cong \KK$ in this case, we can take the basis $\Set{1}$ for $A_{1}$. 
Hence the $0$th Hessian of $\Phi$ with respect to the basis $\Set{1}$ is $\Phi$. 
\end{Remark}

There is a criterion for the strong Lefschetz property for a graded Artinian Gorenstein algebra.  

\begin{Theorem}[Watanabe \cite{W2000}, Maeno--Watanabe \cite{MR2594646}]\label{criterion}
Consider the graded Artinian Gorenstein algebra $A$ with the following  presentation and decomposition: 
$A=\KK[x_{1}, x_{2}, \ldots, x_{N}]/\Ann(\Phi)=\bigoplus_{k=0}^{s} A_{k}$. 
Let $L=a_{1}x_{1}+a_{2}x_{2}+\cdots+a_{N}x_{N}$. 
The multiplication map 
$\times L^{s-2k}\colon A_{k}\to A_{s-k}$ is bijective 
if and only if 
\begin{align*}
\det H_{\Phi}^{(k)}(a_{1}, a_{2}, \ldots, a_{N})\neq 0. 
\end{align*}
\end{Theorem}

For a matroid $M$ with the ground set $E$, the algebra $A_{M}$ is defined by 
\begin{align*}
\KK[x_{e}|e\in E]/\Ann(\Phi_{M}). 
\end{align*}

Theorem \ref{SLP for our algebra (Kn)} follows from Theorems \ref{criterion} and \ref{main theorem}. 

\begin{Theorem}\label{SLP for our algebra (Kn)}
Let $E$ be the edge set of the complete graph $K_{n}$. 
In this case, the ground set of $M_{K_{n}}^{r}$ is $E$. 
Let $N=\#E$, and identify $E$ with $\Set{1, 2, \ldots, N}$. 
Consider the algebra $A_{M_{K_{n}}^{r}}=\bigoplus_{k=0}^{r}A_{k}$ for $2<r<n$. 
Let $L=x_{1}+\cdots+x_{N}$. 
The multiplication map $\times L^{r-2}$ from $A_{1}$ to $A_{r-1}$ is bijective.  
\end{Theorem}

\begin{Corollary}
The algebra $A_{M_{K_{n}}^{r}}$ has the strong Lefschetz property for $n\leq 5$ and $2<r<n$. 
The element $x_{1}+\cdots+x_{N}$ is a strong Lefschetz element. 
\end{Corollary}

Theorem \ref{SLP for our algebra (Kmn)} follows from Theorems \ref{criterion} and \ref{main theorem 2}. 

\begin{Theorem}\label{SLP for our algebra (Kmn)}
Let $E$ be the edge set of the complete bipartite graph $K_{m,n}$. 
In this case, the ground set of $M_{K_{n}}^{r}$ is $E$. 
Let $N=\#E$, and identify $E$ with $\Set{1, 2, \ldots, N}$. 
Consider the algebra $A_{M_{K_{m,n}}^{r}}=\bigoplus_{k=0}^{r}A_{k}$ for $2<r<n$. 
Let $L=x_{1}+\cdots+x_{N}$. 
The multiplication map $\times L^{r-2}$ from $A_{1}$ to $A_{r-1}$ is bijective.  
\end{Theorem}

\begin{Corollary}
The algebra $A_{M_{K_{m,n}}^{r}}$ has the strong Lefschetz property for $n\leq 5$ and $2<r<n$. 
The element $x_{1}+\cdots+x_{N}$ is a strong Lefschetz element. 
\end{Corollary}

Finally, we refer a recent work \cite{MNY}. 

\begin{Remark}\label{MNY}
It is shown that the Hessian of the generating function for simple matroid does not vanish in \cite{MNY}. 
Moreover, its Hessian matrix has exactly one positive eigenvalue. 
In \cite{MNY}, the authors show that the strong Lefschetz property and the Hodge--Riemann relation are equivalent, and all variables are form of a basis for the algebra associated to simple matroid. 
They also show that the strong Lefschetz property of the algebra associated to any matroid by simplifying matroids. 

Our results in this paper are similar to the main result in special case in \cite{MNY}. 
But this paper calculates the eigenvalues of the Hessian matrix concretely, 
and gives another proof of a part of the result in \cite{MNY}. 
\end{Remark}

\bibliographystyle{amsplain-url}
\bibliography{mr}


\end{document}